\documentclass[12pt]{amsart}
\usepackage{latexsym,amssymb,amsmath}
\usepackage{amsmath,amsfonts}
\usepackage{epsfig}
\usepackage{graphicx}
\usepackage{verbatim}
\newtheorem{theorem}{Theorem}[section]
\newtheorem{lemma}[theorem]{Lemma}
\newtheorem{proposition}[theorem]{Proposition}

\newtheorem{rmk}[theorem]{Remark}

\allowdisplaybreaks 

\def\T{\mathbb T}

\def\sign{{\rm sign}}

\def\R{{\mathbb R}}

\def\la{\langle}
\def\ra{\rangle}

\def\mF{{\mathcal F}}

\def\les{\lesssim}
\def\1{{\bf 1}}

\def\eqnn{\begin{eqnarray*}}
\def\eeqnn{\end{eqnarray*}}
\def\eqn{\begin{eqnarray}}
\def\eeqn{\end{eqnarray}}
\newcommand{\nc}{\newcommand}
\nc{\be}{\begin{equation}}
\nc{\ee}{\end{equation}}
\nc{\ba}{\begin{eqnarray}}
\nc{\ea}{\end{eqnarray}}
\nc{\eps}{\epsilon}
\def\prf{\begin{proof}}
\def\endprf{\end{proof}}
\begin{document}

\title[Smoothing of nonlinear waves]{Smoothing estimates for weakly nonlinear internal waves in a rotating ocean}

\author[ Erdo\u{g}an  and Tzirakis]{M. B. Erdo\u{g}an   and N. Tzirakis}
\thanks{The first author is partially supported by NSF grant  DMS--2154031.}
 
\address{Department of Mathematics \\
University of Illinois \\
Urbana, IL 61801, U.S.A.}
\email{berdogan@illinois.edu}

\address{Department of Mathematics \\
University of Illinois \\
Urbana, IL 61801, U.S.A.}
\email{tzirakis@illinois.edu}

\date{}

\begin{abstract}
In this paper we study the effect of rotation on nonlinear wave phenomena in weakly dispersive media modeled by the Korteweg-de Vries equation on the real line. It is well known that smoothing in the case of the KdV equation with periodic boundary conditions is a result of the presence of high frequency waves that weaken the nonlinearity through time averaging, \cite{bit},  \cite{et1}. High frequency interactions through normal form transformations can be thought of as adding a fast rotating term in certain toy PDE models that arrest the development of singularities, \cite{bit}.  It is crucial for this phenomena that  the zero Fourier mode can be removed due to the conservation of the mean. 
On the real line this mechanism breaks down as the resonance sets close to zero frequency are sizable and normal form transformations are not useful, \cite{imt}, and hence smoothing fails. The model we study is a perturbation of the KdV equation on a rotating frame of reference. A combination of a suitable normal form transformation along with the restricted norm method of Bourgain \cite{Bou2} implies smoothing on the real line.
\end{abstract}

\maketitle
\section{Introduction}
On this paper we study a perturbation of the Korteweg-de Vries (KdV) equation that models weakly nonlinear waves in the ocean taking into account the rotation of the earth,  \cite{ostr}, \cite{step}. Unlike KdV, the equation is not known to be completely integrable, see Section 3.2 in \cite{go}. The main purpose of our exposition is to prove that, when the data are in classical Sobolev spaces, the nonlinear part of the equation on the real line is smoother than the initial data. Nonlinear smoothing in dispersive equations has recently played an important role in analyzing the local and global behavior of low regularity solutions \cite{Bbook},  \cite{taobook}, \cite{etbook}.  There exist a variety of applications of nonlinear smoothing in both linear and nonlinear set-ups for dispersive partial differential equations (PDEs) with various boundary conditions. The most common situations are the open boundary case of the real line and  bounded intervals with periodic boundary conditions. Also of interest are  the Euclidean spaces in general, semi-infinite intervals, and bounded intervals where boundary conditions for the values of the data or their derivatives are present. Some 30 years ago through the seminal work of   Bourgain, it was realized that nonlinear smoothing plays an important role in obtaining global-in-time solutions for dispersive equations and systems with infinite energy, \cite{bou2}.  Smoothing was also used to obtain higher order Sobolev norms estimates for globally defined solutions of dispersive equations that do not scatter, \cite{bou1}.

For periodic problems smoothing has been used to study the problem of dispersive quantization (or Talbot effect) in dispersive PDEs, see \cite{et2}, \cite{cet}, \cite{es} and the references therein. Moreover, for both periodic and non-periodic PDEs smoothing leads to the asymptotic compactness of the global evolution, which in turn implies the existence and uniqueness of global attractors for forced and dissipative dispersive PDE, \cite{et4}.  Finally another important application is the almost everywhere convergence of solutions of dispersive equations to the initial data in Sobolev spaces. For similar developments the interested reader can consult the book \cite{etbook}.

In the current manuscript we prove nonlinear smoothing for a variant of the KdV   equation after transforming it using   normal form transformations. More precisely we study the Cauchy problem on $\R$ associated with the equation 
\be \label{Ostr}\left\{\begin{array}{l}
u_t+\beta u_{xxx}- \gamma \partial_x^{-1} u + \partial_x(u^2)  =0,\,\,\,\,x\in\R , t\in \R,\\
u(x,0)=f(x),  
\end{array}\right. 
\ee
 which describes weakly nonlinear wave processes in the ocean taking into account the earth rotation, \cite{ostr}, \cite{step}. In the literature, this equation is often called the Ostrovsky equation. The parameter $\gamma$ measures the effect of rotation and the parameter $\beta$ the strength of the dispersion. Note that the operator $\partial_x^{-1}$ is defined on the Fourier side by
 $$\widehat{\partial_x^{-1}f}(\xi)=\frac{\widehat{f(\xi)}}{i\xi}.$$
 In our paper, we consider the case where $\beta \cdot \gamma>0$. In particular, it is enough to consider  $\beta=  \gamma=1$. Later in an Appendix we show that the nonlinear solution is not smoother than the initial data when $\beta \cdot \gamma<0$. For any $\beta \cdot \gamma \neq 0$ the equation is optimally well--posed at the $H^{-\frac34}$ level, \cite{isa}, \cite{lhy}, \cite{yanyang}. A result of Tsugawa, \cite{tsug} slightly improved the result in \cite{isa} by proving well--posedness in anisotropic Sobolev spaces that contains the $-\frac34$ threshold as a special case. This last result also proves that the solution of the
equation \eqref{Ostr} converges to the solution of the KdV equation
when the rotation parameter $\gamma$ goes to $0$ and the common initial data  is in the $L^2$ space. To prove such a  result the author needed to  obtain a new, uniform in $\gamma$, bilinear
estimate.
 
 Smooth solutions of \eqref{Ostr} satisfy $L^2$ conservation, namely $\|u(t)\|_{L_x^2(\R)}=\|f\|_{L^2(\R)}$. The above results then establish that the solutions are in fact global at least for $L^2$ data. This result was extended in \cite{isa1} to include global-in-time solutions for any $s>-\frac{3}{10}$. 
 
 In \cite{isa}, the authors prove that the initial value problem \eqref{Ostr} is locally well--posed   in $H^{s}$, $s>-\frac34$, by constructing a unique solution $u\in X^{s,b}$ for any $s>-\frac34$ and $b>\frac12$ valid in $[-T,T]$ where $T=T(\|f\|_{H^s})$. In effect, the local solution can be obtained by means of Picard iterations of the Duhamel operator
 $$
u=e^{-tL }f(x)-  e^{-tL} \int_0^t e^{sL}(u u_{x}) ds.
$$ 
Here we prove that the difference of the nonlinear solution with the linear evolution (the integral term above)  is in a smoother Sobolev space than the initial data as long as the solution is locally well--posed by the result in \cite{isa}.  For the original KdV equation smoothing is not possible on the real line, \cite{imt}. To accomplish this we apply the differentiation by parts method of \cite{bit}. It may be  possible to obtain a smoothing estimate by using the sharp result of \cite{isa} within the framework of the $X^{s,b}$ method, but the smoothing gained would be rather marginal, much less than the values we outline here. We now state the main result of this paper: 
\begin{theorem}\label{theo:main1}
Fix $s>-\frac34 $ and $a<\min(s+\frac34,\frac12)$. Consider the real valued solution of the equation \eqref{Ostr} (with $\beta=\gamma=1$)
on $\mathbb{R}\times [0,T_{LWP}]$ with initial data $u(x,0)=f(x) \in H^s(\Bbb R)$. Then $u(t)-e^{tL}f \in C^0_tH^{s+a}_x$ and  
$$\|u(t)- e^{-tL}f\|_{H^{s+a}} \leq  C(s,\|f\|_{H^s}),$$ 
on $[0,T_{LWP}]$ where $L=\partial_{xxx}-\partial_x^{-1}$. 
\end{theorem}
As in \cite{et1}, one can concatenate the smoothing statement and obtain a global smoothing estimate in the full range where global wellposedness is known, $s>-\frac3{10}$. 

To prove Theorem~\ref{theo:main1} we introduce a normal form transformation on \eqref{Ostr} which is similar to the one we applied to the periodic KdV problem in \cite{et1}. We should mention here that although Sobolev space smoothing fails for the KdV equation on $\R$, Colliander et al. in \cite{cst} proved that the nonlinear part of the KdV equation on the real line is smoother than the initial data for solutions that are restricted on high frequencies. Thus in general low--low and high--high interactions are not problematic in KdV type equations.  Therefore, we split the   nonlinearity into two parts 
and estimate them separately. The first part includes  low frequencies and  high-high interactions for which the smoothing can be established without a normal form transformation. These terms provide an upper bound on the smoothing estimate $a=s+\frac34$. As it is usually the case, the smoothing fails at the $-\frac34$ threshold. The harder to estimate terms are in the second part, which includes the low-high interactions (corresponding to the resonant terms for periodic KdV). We handle these terms with the help of a normal form tranformation. 

The main idea in our result rests on the fact that for the perturbation of KdV that we study the dispersive symbol can be lower bounded away from zero (see \eqref{phibound22} below), and thus the smoothing effect takes hold even for the low--high interactions on the real line. For problems with periodic boundary conditions, because of the conservation of the mean, one can remove the zero Fourier mode and hence the low--high interactions do not cause a problem. As such there are now many smoothing results for certain variants of the KdV equation on $\T$. We should also note that, in \cite{CS,CS1}, the authors obtained many smoothing estimates for dispersive PDEs on $\R^d$ with nonlinearities of order 3 and higher. However, their method fails for the quadratic nonlinearities as it is expected from the failure of smoothing in the KdV case. 

For our model the improvement in the estimates come from the phase function $\phi(\xi)=\xi^3-\frac{1}{\xi}$.  After the normal form transformation and because of the nonlinear interactions, the following function is introduced in the denominator of the  multilinear estimates:
\begin{multline}\label{phiid1}
\phi(\xi)-\phi(\xi_1)-\phi(\xi-\xi_1)  =\xi^3-\xi_1^3-(\xi-\xi_1)^3-\frac{1}{\xi}+\frac{1}{\xi_1}+\frac{1}{\xi-\xi_1}\\ =\frac{3\xi^2\xi_1^2(\xi-\xi_1)^2+\xi^2+\xi_1^2-\xi\xi_1}{\xi \xi_1 (\xi-\xi_1)},
\end{multline}
which implies 
\be\label{phibound21}  \frac{1}{|\phi(\xi)-\phi(\xi_1)-\phi(\xi-\xi_1)|}\les \frac{|\xi \xi_1 (\xi-\xi_1)|}{\xi^2\xi_1^2(\xi-\xi_1)^2+\xi^2+\xi_1^2}
\ee
Notice that for the low--high interactions ($|\xi_1|\ll |\xi|$), this term can be further bounded by
\be\label{phibound22} \frac{1}{|\phi(\xi)-\phi(\xi_1)-\phi(\xi-\xi_1)|}\les \frac{| \xi_1 |}{\xi^2\xi_1^2 +1}=K(\xi,\xi_1),
\ee
and that the maximum of the $K$ function is of the order $\frac{1}{\xi}$. This recovers the derivative of the nonlinearity. For the details, see Section 3.

 The paper is organized as follows.  In Section 2, we introduce our notation and define the spaces that we use. In addition, we state an elementary lemma that we use throughout the paper in order to prove the multilinear smoothing estimates. Section 3 introduces the normal form method for   equation \eqref{Ostr} and contains the heart of our arguments where different interactions are estimated. Section 4 contains an Appendix where we show that smoothing fails for equation \eqref{Ostr} for any $\gamma \leq 0$ when $\beta=1$. As we have already mentioned,  the case $\gamma=0$ corresponds to the KdV result of \cite{imt}.
 
\section{Notation} \label{sec:notation}

Recall that for $s\in \Bbb R$, $H^s(\Bbb R)$ is defined as a subspace of $L^2$ via the norm
$$
\|f\|_{H^s(\Bbb R)}:=\sqrt{\int_{\Bbb R} \la \xi \ra^{2s} |\widehat{f}(\xi)|^2\ d\xi},
$$
where $\la \xi \ra:=(1+{\xi}^2)^{1/2}$ and $\widehat{f}(\xi)=\int_{\Bbb R}f(x)e^{-2\pi i x \xi} dx$ are the Fourier coefficients of $f$. Plancherel's theorem takes the form
$$\sum_{\xi \in \R}  |\widehat{f}(\xi)|^2=\int_{\R}|f(x)|^2dx.$$
We denote the linear propagator of the equation as $e^{tL}$, where it is defined on the Fourier side as $\widehat{(e^{tL}f)}(\xi)=e^{ -it \phi(\xi)}\widehat{f}(\xi)$ and $\phi(\xi)=\xi^3-\frac{1}{\xi}$. 
We   also use $(\cdot)^+$ to denote $(\cdot)^\epsilon$ for all $\epsilon>0$ with implicit constants depending on $\epsilon$.

  The Bourgain spaces, $X^{s,b}$, will be defined as the closure of compactly supported smooth functions under the norm $$\|u\|_{X^{s,b}} :=\|e^{-tL}u\|_{H^b_t(\mathbb{R})H^s_x(\mathbb{R})}=\|\langle \tau-\phi(\xi) \rangle^{b} \langle \xi \rangle^s\widehat{u}(\xi,\tau)\|_{L_{\tau}^2 L^2_{\xi}}.$$

We close this section by presenting an elementary lemma that will be used repeatedly. For the proof see \cite{et3}.
\begin{lemma}\label{lem:sums} If $\beta \geq \gamma \geq 0$ and $\beta + \gamma > 1$, then
\[ \int \frac{1}{\langle x-a_1 \rangle ^\beta \langle x - a_2 \rangle^\gamma} dx \lesssim \langle a_1-a_2 \rangle ^{-\gamma} \phi_\beta(a_1-a_2),  \]
where
\[ \phi_\beta(a) \sim \begin{cases} 1 & \beta > 1 \\  \log(1 + \langle a \rangle ) & \beta = 1 \\ \langle a \rangle ^{1-\beta} &\beta < 1 .\end{cases} \]
\end{lemma}

\section{Normal form transform}\label{sec:NFT}
In this section we apply the normal form transformation  to the equation \eqref{Ostr} following the differentiation by parts method of \cite{bit}.  See Proposition~\ref{prop:normal} for the transformed equation. A similar method  was used in our work in \cite{et1} for the periodic KdV  (also see \cite{egt} for an application on $\R$, especially Proposition 6.1). Notice that for more regular data we gain half a derivative in our smoothing estimate, which is sharp (see Remark \eqref{rmk:Bsharp}), while in the case of periodic KdV we gained a full derivative. In both cases the formulas contain boundary terms, resonant terms and non-resonant terms. The non-resonant terms can be further transformed if needed. As noted above, the upper bound in the smoothing estimate for more regular solutions comes from the boundary terms, $B(u)$ in \eqref{Bhat},  which indicates that further applications of the normal form transformation that exist in the literature may not improve the result.  

To perform differentiation by parts we define $L=\partial_{xxx}-\partial_x^{-1}$ and $N(u)=\partial_x(u^2)$, and we write \eqref{Ostr} for $\beta=\gamma=1$ as 
\be\label{LNeq}
u_t+Lu+N(u)=0.
\ee
We observe that smoothing for the low frequency part and the high-high interactions can be obtained without normal-forms. Therefore,  
we further write
$$
N(u)=R(u)+   \widetilde N(u), 
$$
where 
$$
R(u):= P_{|\xi|\les 1} N(u) +P_{|\xi|\gg 1} N(P_{>|\xi|/100} u), 
$$
$$
\mF (\widetilde N(u)) (\xi)=i  \xi \int_\R m(\xi,\xi_1) u(\widehat \xi_1,t) u(\widehat{\xi-\xi_1},t)  d\xi_1 
$$
where 
$$m(\xi,\xi_1) = \chi_{|\xi|\gg 1} \big(\chi_{|\xi_1|\leq |\xi|/100} + \chi_{|\xi-\xi_1|\leq |\xi|/100}\big).$$

Letting $\phi(\xi)=\xi^3-\tfrac1\xi$, we have
$$
e^{Lt}u=\mF^{-1} \big[e^{-i\phi(\xi) t}u(\widehat \xi,t) \big].
$$
The following proposition is the main tool for proving Theorem~\ref{theo:main1}.
 \begin{proposition}\label{prop:normal} The solution $u$ of equation \eqref{Ostr} satisfies 
 \be\label{preduhamelR}
  \partial_t \left(e^{Lt } u - e^{Lt }  B (u)  \right)= - e^{Lt} \big(R(u) +   NR(u)   \big),
\ee
where
\be\label{Bhat}
\widehat{B (u)}(\xi)= \xi \int m(\xi,\xi_1)   \frac{  u_{\xi_1}  u_{\xi-\xi_1}}{\phi(\xi)-\phi(\xi_1)-\phi(\xi-\xi_1)} d\xi_1,  
\ee
\be\label{NRhat}
\widehat{NR (u)}(\xi)= \xi \int m(\xi,\xi_1) \frac{ u_{\xi_1}  w_{\xi-\xi_1}+w_{\xi_1} u_{\xi-\xi_1} }{\phi(\xi)-\phi(\xi_1)-\phi(\xi-\xi_1)} d\xi_1.
\ee
Here $u_\xi(t):=u(\widehat\xi,t)$ and  $w_\xi(t):=w(\widehat\xi,t)$ with 
$$
w=e^{-Lt}\partial_t(e^{Lt} u)= -N(u).
$$ 
\end{proposition}
 
\begin{proof}
The following calculations can be justified  by smooth approximations. First observe that by \eqref{LNeq}
$$
\partial_t(e^{Lt}u)=-e^{Lt}N(u)= -e^{Lt} (R(u)+  \widetilde N(u)).
$$
On the Fourier side, we have
$$
  \mF(e^{Lt}\widetilde N(u))(\xi)=  i \xi \int e^{-i\phi(\xi) t} m(\xi,\xi_1)     u_{\xi_1} u_{\xi-\xi_1}  d\xi_1.
$$
We rewrite the   integral above as
\begin{multline*}
  i \xi \int e^{-i [\phi(\xi)-\phi(\xi_1)-\phi(\xi-\xi_1)] t} m(\xi,\xi_1)     [e^{Lt}u]_{\xi_1} [e^{Lt}u]_{\xi-\xi_1}  d\xi_1\\
= - \xi \partial_t  \int \frac{e^{-i [\phi(\xi)-\phi(\xi_1)-\phi(\xi-\xi_1)] t}}{ [\phi(\xi)-\phi(\xi_1)-\phi(\xi-\xi_1)]} m(\xi,\xi_1)     [e^{Lt}u]_{\xi_1} [e^{Lt}u]_{\xi-\xi_1}  d\xi_1\\
+\xi  \int \frac{e^{-i [\phi(\xi)-\phi(\xi_1)-\phi(\xi-\xi_1)] t}}{ [\phi(\xi)-\phi(\xi_1)-\phi(\xi-\xi_1)]} m(\xi,\xi_1)   \partial_t\big(  [e^{Lt}u]_{\xi_1} [e^{Lt}u]_{\xi-\xi_1}\big)  d\xi_1\\
=-\partial_t \mF (e^{Lt}B(u)) +\mF(e^{Lt} NR(u)).
\end{multline*}
\end{proof}
The following propositions estimate the terms that appear  in \eqref{preduhamelR}.  
\begin{proposition}\label{prop:smooth1} For fixed $s> -\frac34 $,   we have
\begin{align*}  
&\|B(u)\|_{H^{s+a}}\les \|u\|_{H^s}^2,  \,\,\,\,   a\leq \frac12 \\
&\|R(u)\|_{X^{s+a,-\frac12+}}\les \|u\|_{X^{s,\frac12+}}^2, \,\,\,\, a<  \frac34+s.  
\end{align*}  
\end{proposition}
\begin{proof} 
Recalling that $\phi(\xi)=\xi^3-\frac{1}{\xi}$,  we have 
\begin{multline}\label{phiid}
\phi(\xi)-\phi(\xi_1)-\phi(\xi-\xi_1)  =\xi^3-\xi_1^3-(\xi-\xi_1)^3-\frac{1}{\xi}+\frac{1}{\xi_1}+\frac{1}{\xi-\xi_1}\\ =\frac{3\xi^2\xi_1^2(\xi-\xi_1)^2+\xi^2+\xi_1^2-\xi\xi_1}{\xi \xi_1 (\xi-\xi_1)}.
\end{multline}
Therefore,
\be\label{phibound1}  \frac{1}{|\phi(\xi)-\phi(\xi_1)-\phi(\xi-\xi_1)|}\les \frac{|\xi \xi_1 (\xi-\xi_1)|}{\xi^2\xi_1^2(\xi-\xi_1)^2+\xi^2+\xi_1^2}
\ee
We start with the claim for $B$. By symmetry and noting  the support of $m(\xi,\xi_1)$,   we can assume that $|\xi_1|\leq |\xi|/100<|\xi-\xi_1|$. Therefore,
\be\label{phibound} \frac{1}{|\phi(\xi)-\phi(\xi_1)-\phi(\xi-\xi_1)|}\les \frac{| \xi_1 |}{\xi^2\xi_1^2 +1}.
\ee
Using this in \eqref{Bhat}, we have
$$\|B(u)\|_{H^{s+a}}\les \Big\|\xi^{1+s+a} \int m(\xi,\xi_1)   \frac{ |\xi_1| \la \xi_1\ra^{-s} \la \xi-\xi_1\ra^{-s}  f(\xi_1) f(\xi-\xi_1) }{\xi^2\xi_1^2+1 } \ d\xi_1 \Big\|_{L^2},
$$ where $f(\xi)=\la\xi\ra^s |u_\xi|$; $\|f\|_{L^2}=\|u\|_{H^s}$.   

By duality and Cauchy-Schwarz inequality
$$
\|B(u)\|_{H^{s+a}}^2\les \|u\|_{H^s}^4\sup_{|\xi|\gg 1}  \xi^{2+2s+2a} \int_{|\xi_1|\ll|\xi| }    \frac{ |\xi_1|^2 \la \xi_1\ra^{-2s} \la \xi-\xi_1\ra^{-2s}   }{(\xi^2\xi_1^2+1)^2 } \ d\xi_1.
$$
Note that the supremum is bounded by 
$$
\sup_{|\xi|\gg 1}  |\xi|^{2a-1} \int_{|\eta|\ll\xi^2}   \frac{ |\eta|^2 \la \eta/\xi\ra^{-2s}   }{(\eta^2+1)^2 }\ d\eta \les 1,
$$
provided that $a\leq \frac12$, $s>-\frac34$. 

We now consider $R(u)$: 
$$
R(u):= P_{|\xi|\les 1} N(u) +P_{|\xi|\gg 1} N(P_{>|\xi|/100} u).
$$
The first summand satisfies the claim by the local wellposedness theory, for any $a$. 
The contribution of the second summand to the $X^{s+a,-\frac12+}$ norm is bounded by 
$$
\Big\|\chi_{|\xi|\gg1} \int_{|\xi_1|,|\xi-\xi_1|\gtrsim |\xi|}     \frac{|\xi|^{1+s+a} |\xi_1|^{-s} |\xi-\xi_1|^{-s}  f(\xi_1,\tau_1)f(\xi-\xi_1,\tau-\tau_1)}{\la\sigma_1\ra^{\frac12+} \la \sigma_2\ra^{\frac12+} \la\sigma\ra^{\frac12-}} d \xi_1 d \tau_1    \Big\|_{L^2_{\xi,\tau}}, 
$$
where $\|f\|_{L^2_{\xi,\tau}}\les \|u\|_{X^{s,\frac12+}}$. Here $\sigma_1:=  \tau_1-\phi(\xi_1)  $, and $\sigma_2:= \tau-\tau_1-\phi(\xi-\xi_1)  $, and $\sigma:=   \tau -\phi(\xi ) $. We consider only the case when $\rho:=-s\in[0,\frac34)$; the case $s>0$ is easier. Also note that in the domain of the integral, we always have $|\xi-\xi_1|\approx|\xi_1|\gtrsim|\xi|$.

By symmetry, it suffices to consider the cases when $\la\sigma_1\ra$ is the largest and when  $\la\sigma\ra$ is the largest. When $\la\sigma_1\ra$ is the largest, using duality and Cauchy-Schwarz inequality, it suffices to bound  
 $$
\sup_{\xi_1,\tau_1} \int_{1\ll|\xi| \les |\xi_1|\approx|\xi-\xi_1|} \frac{|\xi|^{2+2a-2\rho} |\xi_1|^{2\rho} |\xi-\xi_1|^{2\rho}  d\xi d\tau}{\la\sigma_1\ra^{1+} \la \sigma_2\ra^{1+} \la\sigma\ra^{1-}}.
$$
Integrating in $ \tau$ and using the fact that 
$$\la\sigma_1\ra=\max(\la\sigma_1\ra,\la \sigma_2\ra,\la\sigma\ra)\gtrsim |\phi(\xi)-\phi(\xi_1)-\phi(\xi-\xi_1)|\gtrsim |\xi\xi_1(\xi-\xi_1)| \approx |\xi|\xi_1^2 $$ (the second inequality follows from \eqref{phiid}), we have the bound 
 $$
\sup_{\xi_1,\tau_1} \int_{1\ll|\xi| \les |\xi_1|\approx |\xi-\xi_1|} \frac{|\xi|^{2+2a-2\rho} |\xi_1|^{4\rho}   }{|\xi|^{1+} |\xi_1|^{2+}     \la\tau_1+\phi(\xi-\xi_1)-\phi(\xi)\ra^{1-}} \ d\xi
$$ 
 $$
\les \sup_{\xi_1,\tau_1} \int_{1\ll|\xi| \les |\xi_1|\approx|\xi-\xi_1|} \frac{|\xi|^{1+2a-2\rho-} |\xi_1|^{4\rho-2-}     }{  \la\tau_1+(\xi-\xi_1)^3-\xi^3 \ra^{1-}}\ d\xi.
$$
Letting $\eta= (\xi-\xi_1)^3-\xi^3$, we see that  
$$
d\eta \approx |\xi_1|^{1/2} |\eta+ \xi_1^3/4|^{1/2}  d\xi.
$$
Also note that when $|\xi_1|\gg|\xi|$
$$d\eta \approx |\xi_1|^{2-} |\eta+ \xi_1^3/4|^{0+} d\xi.
$$
In the case $|\xi |\approx |\xi_1|$, we bound the integral by
$$\int  \frac{|\xi_1|^{2a+2\rho-\frac32 }  d\eta }{ |\eta+ \xi_1^3/4|^{1/2} \la \tau_1+\eta\ra^{1-} } \les |\xi_1|^{2a+2\rho-\frac32 }\les 1
$$ 
provided that $a\leq \frac34-\rho$. In the  case $1\ll |\xi |\ll |\xi_1|$, we have the bound 
$$
  \int   \frac{|\xi(\eta)|^{1+2a-2\rho}|\xi_1|^{4\rho-4+} d\eta }{ |\eta+ \xi_1^3/4|^{0+} \la \tau_1+\eta\ra^{1-} } \les  \int \frac{|\xi(\eta)|^{-3+2a+2\rho+}  d\eta }{|\eta+ \xi_1^3/4|^{0+}  \la \tau_1+\eta\ra^{1+} } \les 1
$$
provided that $a< \frac32-\rho$. 

The case $\la\sigma\ra$ is the largest is analogous by integrating in $\xi_1,\tau_1$ instead of $\xi,\tau$. 
\end{proof}
\begin{rmk}\label{rmk:Bsharp}
 We note that the regularity of the $B$ term in Proposition~\ref{prop:smooth1} cannot be improved for any $s$ by considering the example $f(\xi_1)=N^\frac12\chi_{[\frac1N,\frac2N]}(\xi_1)+\chi_{[N,N+1]}(\xi_1)$, $N\gg 1$.
\end{rmk} 
\begin{proposition}\label{prop:smooth2} For fixed $s> -\frac34 $ and  $a<\min(\frac34+s,\frac12)$,  we have
$$
\big\| NR(u) \big\|_{X^{ s+ a,-\frac12+ }} \les \|u\|_{X^{s,\frac12+}}^3 
$$
\end{proposition}
\begin{proof} 
Recall that 
$$m(\xi,\xi_1) = \chi_{|\xi|\gg 1} \big(\chi_{|\xi_1|\leq |\xi|/100} + \chi_{|\xi-\xi_1|\leq |\xi|/100}\big).$$
By $\xi_1\leftrightarrow\xi-\xi_1$ symmetry, it suffices to consider only the contribution of the first summand of $m$ to $NR(u) $:
\begin{multline*}
\xi \int m(\xi,\xi_1) \frac{ u_{ \xi_1}  w_{ \xi-\xi_1}  +w_{\xi_1}u_{\xi-\xi_1} }{\phi(\xi)-\phi(\xi_1)-\phi(\xi-\xi_1)} \ d\xi_1
=\\
\xi \chi_{|\xi|\gg 1}  \int _{|\xi_1|\leq |\xi|/100}   \frac{(\xi- \xi_1)\ u_{ \xi_1}  u_{\xi_2}  u_{\xi-\xi_1-\xi_2} + \xi_1 u_{\xi-\xi_1}u_{\xi_2} u_{\xi_1-\xi_2}}{\phi(\xi)-\phi(\xi_1)-\phi(\xi-\xi_1)} \ d\xi_1\ d\xi_2.
\end{multline*} 
Below, we provide details for the estimate of the first summand in the numerator. As $\xi_1$ is smaller compared to $\xi-\xi_1$, the analysis of the second term is similar and   easier. 

Letting $\rho=-s \in [0,3/4)$ and by duality, it suffices to bound the integral below by $\|f\|_{L^2_{\xi,\tau}}^4$ 
$$
 \int_{|\xi|\gg 1, |\xi_1|\leq |\xi|/100}  \frac{|\xi|^{2 -\rho+a}  \la  \xi_1\ra^\rho \la \xi_2\ra^\rho \la \xi-\xi_1-\xi_2\ra^\rho f_{ \xi_1}f_{\xi_2}  f_{\xi-\xi_1-\xi_2} f_\xi}{|\phi(\xi)-\phi(\xi_1)-\phi(\xi-\xi_1)| \la\sigma_1\ra^{\frac12+}\la \sigma_2\ra^{\frac12+}\la\sigma_3\ra^{\frac12+}\la\sigma\ra^{\frac12-}} d\xi_1d\xi_2 d\xi d\tau_1
 d\tau_2d\tau.
$$
Here $\sigma$'s  are defined as in the proof of Proposition~\ref{prop:smooth1}. Also using \eqref{phibound}, we have the bound
\be\label{eq:mainterm} \int_{|\xi|\gg 1, |\xi_1|\leq |\xi|/100}  \frac{ | \xi|^{2-\rho +a} |\xi_1| \la\xi_1\ra^\rho   \la \xi_2\ra^\rho \la \xi-\xi_1-\xi_2\ra^\rho |f_{ \xi_1}f_{\xi_2}  f_{\xi-\xi_1-\xi_2} f_\xi|}{[\xi^2\xi_1^2+1]\ \la\sigma_1\ra^{\frac12+}\la \sigma_2\ra^{\frac12+}\la\sigma_3\ra^{\frac12+}\la\sigma\ra^{\frac12-}} d\xi_1d\xi_2 d\xi d\tau_1
 d\tau_2d\tau.
\ee
Note that,
\begin{multline}\label{eq:M}
\la \sigma-\sigma_1-\sigma_2-\sigma_3\ra=\la  \phi( \xi_1)+\phi(\xi_2)+\phi(\xi-\xi_1-\xi_2)-\phi(\xi)\ra\\
\approx\big\la 3(\xi-\xi_1)(\xi-\xi_2)(\xi_1+\xi_2)+\frac1{\xi_1}+\frac1{\xi_2}+\frac1{\xi-\xi_1-\xi_2} \big \ra:=\la M\ra.
\end{multline}
We will consider the following cases: 
\begin{itemize}
\item Case 1) $|\xi_1|\les 1$, $|\xi_2|\les 1$.
\item Case 2) $1\ll |\xi_1|\ll |\xi|$ and  $|\xi_2|\les 1$.
\item Case 3) $1\ll |\xi_1|\ll |\xi|$,  $1\ll |\xi_2|\les |\xi| $ and $|\xi-\xi_1-\xi_2|\gg 1$.  
\item Case 4) $|\xi_1|\les \frac1{\la \xi (\xi-\xi_2)\xi_2\ra }$  and $|\xi_2|, |\xi-\xi_1-\xi_2|\gg 1$. 
\item Case 5) $1\gtrsim |\xi_1|\gg \frac1{\la \xi (\xi-\xi_2)\xi_2\ra }$  and $|\xi_2|, |\xi-\xi_1-\xi_2|\gg 1$. 
\item Case 6) $1\ll |\xi_1|\ll |\xi| \ll |\xi_2|  $.
\end{itemize}

Case 1) $|\xi_1|\les 1$, $|\xi_2|\les 1$.   By the change of variable $\xi_2\to \xi-\xi_1-\xi_2$, this also takes care of the case $|\xi_1|, |\xi-\xi_1-\xi_2| \les 1$.  By Cauchy-Schwarz inequality, it suffices to prove that
$$
\sup_{|\xi|\gg 1 } \ \int_{  |\xi_1|,|\xi_2|\les 1 }  \frac{ | \xi|^{4  +2a} |\xi_1|^2    }{[\xi^2\xi_1^2+1]^2 \  \la M\ra^{1-}} d\xi_1d\xi_2 < \infty.
$$  
Note that 
$$
\la M\ra\approx \la 3(\xi-\xi_1)(\xi-\xi_2)(\xi_1+\xi_2)+\frac1{\xi_1}+\frac1{\xi_2}  \big \ra =\big\la (\xi_1+\xi_2) [3(\xi-\xi_1)(\xi-\xi_2)+\frac1{\xi_1\xi_2}]\big\ra.
$$
Letting $\eta=(\xi_1+\xi_2) [3(\xi-\xi_1)(\xi-\xi_2)+\frac1{\xi_1\xi_2}]$ in the $\xi_2$ integral, we see that
$$d\eta=[3(\xi-\xi_1)(\xi-2\xi_2-\xi_1)-\frac1{\xi_2^2}] d\xi_2.$$
Note that the Jacobian is $\gtrsim \xi^2$ if $|\xi_2|\not\approx \frac1{|\xi|}$.  In this case we estimate the $\xi_2$ integral by
$$
\Big(\int_{  |\xi_2|\les 1}  \la M\ra^{-1-} d\xi_2\Big)^{1-} \les |\xi|^{-2+} \Big(\frac{1  }{  \la s\ra^{1+}}  ds\Big)^{1-}\les |\xi|^{-2+}.
$$
This leads to the bound
$$
\int_{  |\xi_1| \les 1 }  \frac{ | \xi|^{2  +2a+} |\xi_1|^2    }{[\xi^2\xi_1^2+1]^2 } d\xi_1 \les |\xi|^{-1+2a+},
$$
which is bounded on the set $|\xi|>1$ provided that $a<\frac12$. 

When $|\xi_2| \approx \frac1{|\xi|}$, consider the cases $|\xi_1| \not\approx \frac1{|\xi|}$ and $|\xi_1|  \approx \frac1{|\xi|}$ seperately. In the former case, $\la M\ra\gtrsim \la (\xi_1+\xi_2) \xi^2 \ra$, which leads to (for $a<\frac12$)
$$
\int_{ 1\gtrsim |\xi_1| \not\approx \frac1{|\xi|} ,|\xi_2| \approx \frac1{|\xi|} }  \frac{ | \xi|^{2  +2a+} |\xi_1|^2    }{[\xi^2\xi_1^2+1]^2 \   |\xi_1+\xi_2|^{1-} } d\xi_1d\xi_2\les |\xi|^{-1+2a+}\les 1.
$$
In the latter case, we have $\la M\ra\gtrsim \la (\xi_1+\xi_2)(3 \xi^2+\frac1{\xi_1\xi_2}) \ra\approx  \la (\xi_1+\xi_2) \xi^3 (\xi_2 +\frac1{3\xi_1\xi^2 }) \ra$, which leads to
\begin{multline*}
\int_{  |\xi_1|,|\xi_2| \approx \frac1{|\xi|} }  \frac{ | \xi|^{-1 +2a+ }     }{    |\xi_1+\xi_2|^{1-} | \xi_2+\frac1{3\xi_1\xi^2}|^{1-}} d\xi_1d\xi_2\les
\int_{  |\xi_1|  \approx \frac1{|\xi|} }  \frac{ | \xi|^{-1 +2a+ }     }{    |\xi_1-\frac1{3\xi_1\xi^2}|^{1-}} d\xi_1\\ \les 
\int_{  |\xi_1|  \approx \frac1{|\xi|} }  \frac{ | \xi|^{-2 +2a+ }     }{    |\xi_1^2-\frac1{3 \xi^2}|^{1-}} d\xi_1
\les  |\xi|^{-1+2a+}\les 1 
\end{multline*}
provided that $a<\frac12$.

Case 2) $1\ll |\xi_1|\ll |\xi|$ and  $|\xi_2|\les 1$.   By the change of variable $\xi_2\to \xi-\xi_1-\xi_2$, this also takes care of the case $1\ll |\xi_1|\ll |\xi|$, $|\xi-\xi_1-\xi_2| \les 1$. By Cauchy-Schwarz inequality
$$
\sup_{|\xi|\gg 1 } \ \int_{  1\ll |\xi_1|\ll |\xi|,\ |\xi_2|\les 1 }  \frac{ | \xi|^{4  +2a} |\xi_1|^{2+2\rho} }{[\xi^2\xi_1^2+1]^2 \  \la M\ra^{1-}} d\xi_1d\xi_2 \les \sup_{|\xi|\gg 1 } \ \int_{  1\ll |\xi_1|\ll |\xi|,\ |\xi_2|\les 1 }  \frac{ | \xi|^{ 2a} |\xi_1|^{ 2\rho-2} }{   \la M\ra^{1-}} d\xi_1d\xi_2.
$$  
Note that 
$$
\la M\ra\approx \la 3(\xi-\xi_1)(\xi-\xi_2)(\xi_1+\xi_2)+ \frac1{\xi_2}  \big \ra  
$$
Note that the contribution of the set  $|\xi_2|\les \frac1{ \xi^2 |\xi_1| }$ is $\les 1$ provided that $a<1$.  
  If $|\xi_2|\gg\frac1{ \xi^2 |\xi_1| } $, then $\la M\ra\approx \xi^2 |\xi_1|$, which leads to  the bound
$$\sup_{|\xi|\gg 1 } \ \int_{  1\ll |\xi_1|\ll |\xi|,\ |\xi_2|\les 1 }   | \xi|^{ 2a-2+} |\xi_1|^{ 2\rho-3+}   d\xi_1d\xi_2 \les 1,
$$ 
provided that $a<1$.

Case 3) $1\ll |\xi_1|\ll |\xi|$,  $1\ll |\xi_2|\les |\xi| $ and $|\xi-\xi_1-\xi_2|\gg 1$. Once again we have
$$
\sup_{|\xi|\gg 1 } \ \int_{  1\ll |\xi_1|\ll |\xi|,\ 1\ll |\xi_2|\les |\xi|,\  |\xi-\xi_1-\xi_2|\gg 1 }  \frac{ | \xi|^{-2\rho  +2a} |\xi_1|^{ 2\rho-2} |\xi_2|^{2\rho}  |\xi-\xi_1-\xi_2|^{2\rho} }{   \la M\ra^{1-}} d\xi_1d\xi_2$$
$$
\les \sup_{|\xi|\gg 1 } \ \int_{  1\ll |\xi_1|\ll |\xi|  }  \frac{ | \xi|^{ 2\rho  +2a} |\xi_1|^{ 2\rho-2}   }{  |\xi|^{1-} |\xi-\xi_2|^{1-}  |\xi_1+\xi_2|^{1-}} d\xi_1d\xi_2$$
$$
\les \sup_{|\xi|\gg 1 } \ \int_{  1\ll |\xi_1|\ll |\xi|  }  \frac{ | \xi|^{ 2\rho  +2a-1+} |\xi_1|^{ 2\rho-2}   }{      |\xi_1+\xi |^{1-}} d\xi_1 \les 1 $$
provided that $a<1-\rho$.

Case 4) $|\xi_1|\les \frac1{\la \xi (\xi-\xi_2)\xi_2\ra }$  and $|\xi_2|, |\xi-\xi_1-\xi_2|\gg 1$.

By Cauchy-Schwarz inequality as above 
\begin{multline*}
\sup_{|\xi|\gg 1 } \ \int_{  |\xi_1|\les \frac1{|\xi (\xi-\xi_2) \xi_2| }  }   | \xi|^{4 -2\rho +2a} |\xi_1|^2  \la \xi_2\ra^{2\rho} \la\xi-\xi_2\ra^{2\rho}   d\xi_1d\xi_2 \\ \les  \sup_{|\xi|\gg 1 } \int   \frac{ | \xi|^{4 -2\rho +2a}   \la \xi_2\ra^{2\rho} \la\xi-\xi_2\ra^{2\rho}  }{|\xi (\xi-\xi_2) \xi_2|^3  }  \ d\xi_2 \les \sup_{|\xi|\gg 1 } |\xi|^{2a-2} < \infty,
\end{multline*}
provided that $a\leq 1$.

Case 5) $1\gtrsim |\xi_1|\gg \frac1{\la \xi (\xi-\xi_2)\xi_2\ra }$  and $|\xi_2|, |\xi-\xi_1-\xi_2|\gg 1$.\\ 
Note that in this case, we also have $  |\xi-\xi_2|\gg 1$ and
$$
\la M\ra \approx \la 3(\xi-\xi_1)(\xi-\xi_2)(\xi_1+\xi_2)+\frac1{\xi_1}  \big \ra \approx |\xi(\xi-\xi_2)\xi_2|. 
$$
We bound \eqref{eq:mainterm}  on this set by
$$ \int_{|\xi|\gg 1 \gtrsim |\xi_1|\gg \frac1{\la \xi (\xi-\xi_2)\xi_2\ra }}  \frac{ | \xi|^{2-\rho +a} |\xi_1|    | \xi_2|^\rho | \xi- \xi_2|^\rho |f_{ \xi_1}f_{\xi_2}  f_{\xi-\xi_1-\xi_2} f_\xi|}{[\xi^2\xi_1^2+1]\ \la\sigma_1\ra^{\frac12+}\la \sigma_2\ra^{\frac12+}\la\sigma_3\ra^{\frac12+}\la\sigma\ra^{\frac12-}} d\xi_1d\xi_2 d\xi d\tau_1
 d\tau_2d\tau.
$$
By symmetry, and since $ \max( \la\sigma_1\ra,\la \sigma_2\ra,\la\sigma_3\ra,\la\sigma\ra)\gtrsim \la M\ra$, it suffices to consider the subcases  $\la\sigma_1\ra\gtrsim \la M\ra$, $\la \sigma_2\ra\gtrsim \la M\ra$, and $\la\sigma\ra\gtrsim \la M\ra$. 
 
Subcase 1: $\la\sigma_1\ra\gtrsim \la M\ra\approx |\xi(\xi-\xi_2)\xi_2|$.
By Cauchy-Schwarz inequality and then integrating in $\tau, \tau_2$, and using $\frac{|\xi_1|}{\xi^2\xi_1^2+1}\les \frac1{|\xi|}$, it suffices to consider  
\be\label{xixieta}
\sup_{|\xi_1|\les 1, \tau_1 } \int_{|\xi|\gg 1 \gtrsim |\xi_1|\gg \frac1{\la \xi (\xi-\xi_2)\xi_2\ra }}  \frac{ | \xi|^{1-2\rho +2a}     | \xi_2|^{2\rho-1} | \xi- \xi_2|^{2\rho-1} }{   \la \tau_1+\phi(\xi_2)+\phi(\xi-\xi_1-\xi_2)-\phi(\xi)\ra^{1+}}  d\xi_2 d\xi.
\ee
Noting that the factor in the denominator is $\approx   \la \tau_1+ \xi_2^3+(\xi-\xi_1-\xi_2)^3-\xi^3\ra$, we let $\eta =\xi_2^3+(\xi-\xi_1-\xi_2)^3-\xi^3$. We have 
$$
\Big|\frac{\partial\eta}{\partial\xi_2}\Big|\approx |\xi||\xi-\xi_1-2\xi_2|,\,\,\,\,\Big|\frac{\partial\eta}{\partial\xi }\Big|\approx |\xi_2||2\xi-\xi_1-\xi_2|.
$$
When $|\xi_2|\geq \frac32 |\xi|$, we use the change of variable $\eta$ in the $\xi_2$ integral in \eqref{xixieta}, and obtain 
$$
\les \sup_{|\xi_1|\les 1, \tau_1 } \int_{|\xi|\gg 1 }  \frac{ | \xi|^{ -2\rho +2a}     | \xi_2(\eta,\xi) |^{4\rho-3}   }{   \la \tau_1+\eta\ra^{1+}}  d\eta d\xi
\les   \int_{|\xi|\gg 1 }   | \xi|^{  2\rho +2a-3} d\xi \les 1,
$$
provided that $\rho\leq \frac34$ and $a\leq 1-\rho$. 
 
When $|\xi |\geq \frac32 |\xi_2|$,  we use the change of variable $\eta$ in the  $\xi $ integral: 
$$
\les \sup_{|\xi_1|\les 1, \tau_1 } \int_{|\xi_2|\gg 1}   \frac{ | \xi(\eta,\xi_2)|^{  2a-1}     | \xi_2 |^{2\rho-2}   }{   \la \tau_1+\eta\ra^{1+}}  d\eta d\xi_2
\les   \int_{|\xi_2|\gg 1}   | \xi_2|^{  2\rho +2a-3} d\xi_2 \les 1,
$$
provided that $a\leq \min(\frac12,1-\rho)$.

When $\frac23|\xi|\leq |\xi_2|\leq \frac32|\xi|$, we use the change of variable $\eta$ in the  $\xi $ integral: 
 $$
\les \sup_{|\xi_1|\les 1, \tau_1 } \int_{|\xi_2|\gg 1}   \frac{     | \xi_2 |^{2a-2}   | \xi(\eta,\xi_2) -\xi_2|^{  2\rho-1} }{   \la \tau_1+\eta\ra^{1+}}  d\eta d\xi_2 
\les   \int_{|\xi_2|\gg 1}   | \xi_2|^{  2a -2 +\max(2\rho-1,0)} d\xi_2 \les 1,
$$
provided that $a\leq \frac12-\max(\rho-\frac12,0)$. In the second inequality we used the bound $|\xi-\xi_2|\gg 1$, which is valid throughout the Case 5.  

Subcase 2:  $\la\sigma_2\ra\gtrsim \la M\ra\approx |\xi(\xi-\xi_2)\xi_2|$.
By Cauchy-Schwarz inequality and then integrating in $\tau, \tau_1$, and using $\frac{|\xi_1|^2}{[\xi^2\xi_1^2+1]^2}\les \frac1{|\xi|^{3-}|\xi_1|^{1-}}$  it suffices to consider  
$$
\sup_{|\xi_2|\gg 1, \tau_2 } \int_{|\xi|\gg 1 \gtrsim |\xi_1|\gg \frac1{\la \xi (\xi-\xi_2)\xi_2\ra }}  \frac{ | \xi|^{ -2\rho +2a+}     | \xi_2|^{2\rho-1} | \xi- \xi_2|^{2\rho-1} }{ |\xi_1|^{1-} \la \tau_2+\phi(\xi_1)+\phi(\xi-\xi_1-\xi_2)-\phi(\xi)\ra^{1+}}  d\xi_1 d\xi.
$$
Noting that the factor in the denominator is $\approx   \la \tau_2+ \phi(\xi_1) +(\xi-\xi_1-\xi_2)^3-\xi^3\ra$, we let $\eta =\tau_2+ \phi(\xi_1) +(\xi-\xi_1-\xi_2)^3-\xi^3$. We have 
$$
 \Big|\frac{\partial\eta}{\partial\xi }\Big|\approx |\xi_2||2\xi-\xi_1-\xi_2|.
$$
We use this change of variable when $|2\xi-\xi_2|\gg 1$ to obtain
$$
\sup_{|\xi_2|\gg 1, \tau_2 } \int_{|\xi|\gg 1 \gtrsim |\xi_1| }  \frac{ | \xi|^{ -2\rho +2a+}      | \xi_2|^{2\rho-2} | \xi- \xi_2|^{2\rho-1} }{  |2\xi-\xi_2| |\xi_1|^{1-}  \la \eta \ra^{1+}}  d\eta d\xi_1 \les 1
$$
provided that
$$
\sup_{|\xi|, |\xi_2|,|\xi-\xi_2|,|2\xi-\xi_2| \gg 1} \frac{ | \xi|^{ -2\rho +2a+}      | \xi_2|^{2\rho-2} | \xi- \xi_2|^{2\rho-1} }{  |2\xi-\xi_2| } \les 1.
$$
This inequality holds for $a<1-\rho$ by considering the cases $|\xi|\approx|\xi_2|$ and $|\xi|\not\approx |\xi_2|$ seperately. 

When $|2\xi-\xi_2|\les 1$, we have $|\xi_2|, |\xi-\xi_2|\approx |\xi|$. Therefore, we have the bound
$$
\sup_{|\xi_2|\gg 1, \tau_2 } \int_{|2\xi-\xi_2|, |\xi_1| \les 1 }  \frac{ | \xi|^{  2\rho +2a-2+}     }{ |\xi_1|^{1-}  }  d\xi_1 d\xi\les 1
$$
provided that $a<1-\rho$ (since the $\xi$ integral is on an interval of length $\les 1$). 

Subcase 3:  $\la\sigma \ra\gtrsim \la M\ra\approx |\xi(\xi-\xi_2)\xi_2|$.
By Cauchy-Schwarz inequality and then integrating in $\tau_1, \tau_2$, and using $\frac{|\xi_1|^2}{[\xi^2\xi_1^2+1]^2}\les \frac1{|\xi|^{3-}|\xi_1|^{1-}}$  it suffices to consider  
$$
\sup_{|\xi |\gg 1, \tau  } \int_{ 1 \gtrsim |\xi_1| }  \frac{ | \xi|^{ -2\rho +2a+}     | \xi_2|^{2\rho-1+} | \xi- \xi_2|^{2\rho-1+} }{ |\xi_1|^{1-} \la \tau -\phi(\xi_1)-\phi(\xi-\xi_1-\xi_2)-\phi(\xi_2)\ra^{1+}}  d\xi_1 d\xi_2.
$$
This can be handled as in Subcase 2. 

Case 6) $1\ll |\xi_1|\ll |\xi| \ll |\xi_2|  $.

Note that in this case, we   have  
$$
\la M\ra \approx \la 3(\xi-\xi_1)(\xi-\xi_2)(\xi_1+\xi_2)  \big \ra \approx |\xi \xi_2^2|. 
$$
We bound \eqref{eq:mainterm}  on this set by
$$\int_{  1\ll |\xi_1|\ll |\xi| \ll |\xi_2| }    \frac{ | \xi|^{a-\rho } |\xi_1|^{\rho-1}    | \xi_2|^{2\rho}  |f_{ \xi_1}f_{\xi_2}  f_{\xi-\xi_1-\xi_2} f_\xi|}{  \la\sigma_1\ra^{\frac12+}\la \sigma_2\ra^{\frac12+}\la\sigma_3\ra^{\frac12+}\la\sigma\ra^{\frac12-}} d\xi_1d\xi_2 d\xi d\tau_1
 d\tau_2d\tau.
$$
We consider the  subcases as in Case 5. Since they can be handled similarly, we give details only for subcase 1:

Subcase 1: $\la\sigma_1\ra\gtrsim \la M\ra\approx |\xi \xi_2^2|$.

As above, it suffices to bound
$$
\sup_{|\xi_1|\gg 1, \tau_1} \int_{  |\xi_1|\ll |\xi| \ll |\xi_2| }    \frac{ | \xi|^{2a-2\rho-1 } |\xi_1|^{2\rho-2}    | \xi_2|^{4\rho-2}   }{   \la \tau_1+\phi(\xi_2)+\phi(\xi-\xi_1-\xi_2)-\phi(\xi)\ra^{1+}}  d\xi_2 d\xi.  
$$
Letting $\eta=\tau_1+\xi_2^3+(\xi-\xi_1-\xi_2)^3-\xi^3$ we see that 
$$
 \Big|\frac{\partial\eta}{\partial\xi }\Big|\approx |\xi_2|^2.
$$
This leads to the bound 
 $$
\int_{   |\xi| \ll |\xi_2| }    \frac{ | \xi|^{2a-2\rho-1 }   |\xi_2|^{ 4\rho-4}   }{   \la \eta\ra^{1+}}  d\xi_2 d\eta \les 
\sup_{|\xi|\gg 1}  | \xi|^{2a+2\rho-4+ }   \int_{  |\xi_2| \gg 1  }    \frac{    |\xi_2|^{ -1-}   }{   \la \eta\ra^{1+}}  d\xi_2 d\eta \les 1
$$
provided that $\rho<\frac34$ and $a<2-\rho$. 
\end{proof}

We now outline the proof of Theorem~\ref{theo:main1}. For more details and the global version of the theorem,  see \cite{et1}.
\begin{proof}[Proof of Theorem~\ref{theo:main1}]
 Integrating  \eqref{preduhamelR} on $[0,t]$ we obtain
$$
 u(t)-e^{-Lt}f =B(u(t)) -e^{Lt}B(f) - \int_0^te^{-L(t-s) } \big(R(u) +  NR (u)  \big)\ ds.
$$
The claim follows from this identity using  the bounds in Proposition~\ref{prop:smooth1}, Proposition~\ref{prop:smooth2}, standard $X^{s,b} $ inequalities, and the embedding $X^{s,b}\subset C^0_t H^{s}_x$ for $b>\frac{1}{2}$.
 \end{proof}

\section{Appendix}
Consider the equation 
\be \label{OstrGamma}\left\{\begin{array}{l}
u_t+u_{xxx} -\gamma \partial_x^{-1} u + \partial_x(u^2)  =0,\,\,\,\,x\in\R , t\in \R,\\
u(x,0)=f(x).  
\end{array}\right. 
\ee
The Duhamel formulation of the equation is 
$$
u=e^{-Lt}f(x)-  e^{-Lt} \int_0^t e^{Ls}N(u) ds,
$$
where $Lu=u_{xxx}-\gamma \partial_x^{-1} u$ and $N(u)=\partial_x(u^2) $. Now we consider the nonlinear part of the first Picard iterate, which is obtained by replacing $u$ with $e^{-Lt}f(x)$. On the Fourier side, it is given by
\begin{multline}\label{ptg}
\widehat{P_t(f)}:=-e^{i\phi(\xi)t} \xi  \int_{\R} \int_0^te^{-i [\phi(\xi)-\phi(\eta)-\phi(\xi-\eta)]s}  \widehat{f}(\eta)\widehat{f}(\xi-\eta) ds d\eta
\\
=-ie^{i\phi(\xi)t} \xi \int_{\R}  \frac{e^{-i[\phi(\xi)-\phi(\eta)-\phi(\xi-\eta)]t}-1}{\phi(\xi)-\phi(\eta)-\phi(\xi-\eta)}  \widehat{f}(\eta)\widehat{f}(\xi-\eta)  d\eta, 
\end{multline}
where $\phi(\xi)=\xi^3-\frac{\gamma}{\xi}$.

We claim that there is no smoothing in the first Picard iterate when $\gamma \leq 0$, namely $P_t(f)$ is not smoother than the initial data. This is already known for $\gamma=0$, see \cite{imt}.
We provide the proof for the cases $\gamma=0$ and $\gamma=-1$ for completeness.   
 
For $\gamma=0$, we rewrite \eqref{ptg} as 
$$
\widehat{P_t(f)}= -ie^{i \xi^3 t} \int \frac{\widehat f(\eta) \widehat f(\xi-\eta)}{\eta(\xi-\eta)} (e^{-it3\xi\eta(\xi-\eta)}-1) d\eta.  
$$
Let $\widehat f_N(\eta)=i\ \sign(\eta) \big[N \chi_{[\frac{\alpha}{N^2},\frac{\beta}{N^2}]}(|\eta|)+ N^{-s} \chi_{[N,N+1]}(|\eta|)\big]$. Note that $f_N$ is real-valued. Here $\alpha<\beta$ are chosen depending on fixed $t$ so that
$$
[3t\alpha,3t\beta]\subset [\pi-\frac1{10},\pi+\frac1{10}].
$$ 
This guarantees that $\Re (1-e^{-it3\xi\eta(\xi-\eta)} )\gtrsim 1$ in the support of $\widehat f_N$.  We also have $\|f_N\|_{H^s}\approx 1$.

Note that there are four hi-low contributions and   they have the same sign because of  symmetry and because $\frac{\widehat{f}(\eta)}{\eta}$ is even.  Therefore, for large $N$ and $\xi\in[N,N+1]$, we have 
$$
|\widehat{P_t(f_N)}(\xi)|\approx \int_{[\frac{\alpha}{N^2},\frac{\beta}{N^2}]} N^{1-s} Nd\eta \approx N^{-s}.
$$
We conclude that, for $a>0$,  $\sup_{N } \|P_t(f_N)\|_{H^{s+a}} =\infty$. 

When $\gamma=-1$, we instead have
\be\label{ptggamma}
\big|\widehat{P_t(f)}\big|=\Big| \int \frac{\xi \widehat f(\eta) \widehat f(\xi-\eta)}{\Phi(\xi,\eta)} (e^{-it \Phi(\xi,\eta) }-1) d\eta\Big|,  
\ee
where 
\be\label{bigphi}
\Phi(\xi,\eta) =\phi(\xi)-\phi(\eta)-\phi(\xi-\eta)= 3\xi\eta(\xi-\eta)+\frac1\xi-\frac1\eta-\frac1{\xi-\eta}.
\ee
Taking $N\gg1$ and  
$$\widehat f_N(\eta)=i\sign(\eta) \big[N \chi_{[\frac{1}{\sqrt 3 N -\beta},\frac{1}{\sqrt 3 N -\alpha}]}(|\eta|)+ N^{-s} \chi_{[N,N+1]} (|\eta|)\big],$$
we see that $f_N$ is real-valued and  $\|f_N\|_{H^s}\approx 1$. This is because
$$
\frac{1}{\sqrt 3 N -\alpha}-\frac{1}{\sqrt 3 N -\beta}\les \frac1{N^2}.
$$

Fixing  $\xi\in[N-c,N+c]$, by $\eta\leftrightarrow \xi-\eta$ symmetry in \eqref{ptggamma} and \eqref{bigphi}, we can assume $| \eta| \les \frac1N $.  Therefore, we see that
$$
\Phi(\xi,\eta) =3\xi^2\eta -\frac1\eta+O(1/N).  
$$
We now put additional restrictions on  $c$ and on    $0< \alpha<\beta$ depending only on $t>0$ 
so that for $\xi\in [N-c,N+c]$ and $|\eta|\in [\frac{1}{\sqrt 3 N -\beta},\frac{1}{\sqrt 3 N -\alpha}]$, we have  $$\Phi(\xi,\eta)\ \sign(\eta)\approx 1$$ and $$|t \Phi(\xi,\eta)| \in  [k\pi-\frac1{10},k\pi+\frac1{10}],$$
where $k=k(t)>0$ is odd. 

A calculation as above yields the claim:
$$
|\widehat{P_t(f_N)}(\xi)|\approx \int_{[\frac{1}{\sqrt 3 N -\beta},\frac{1}{\sqrt 3 N -\alpha}]} N^{1-s} Nd\eta \approx N^{-s}.
$$
Therefore, we again conclude that, for $a>0$,  $\sup_{N } \|P_t(f_N)\|_{H^{s+a}} =\infty$.


\begin{thebibliography}{100}

\bibitem{bit} A. Babin, A. A. Ilyin, and E. S. Titi, On the regularization mechanism for the
periodic Korteweg–de Vries equation, Comm. Pure Appl. Math. 64 (2011), no.
5, 591--648.


\bibitem{Bou2} J.~Bourgain, {\it Fourier transform restriction phenomena for certain lattice subsets and
applications to nonlinear evolution equations. Part II: The KdV equation,} GAFA  { 3} (1993), 209--262.

\bibitem{Bbook} J.~Bourgain, {\it Global solutions of nonlinear Schr\"odinger equations}, AMS Colloquium Publications  { 46} (1998).

\bibitem{bou1} J.~Bourgain,  { On the growth in time of higher Sobolev norms of smooth solutions of Hamiltonian PDE,}  Internat. Math. Res. Notices  (1996),  no. 6, 277--304.

\bibitem{bou2} J.~Bourgain, { Refinement of Strichartz' inequality and applications to 2D--NLS with critical nonlinearity,} Internat. Math. Res. Notices (1998), no. 5, 253--283.

\bibitem{cet} V. Chousionis, M. B. Erdo\u gan, and N. Tzirakis, {\it Fractal solutions of linear and nonlinear dispersive partial differential equations,} Proceedings of the London Mathematical Society 110 (3), 543--564.

\bibitem{CS} S. Correia, F. Oliveira, and J. D. Silva, {\it Sharp local well--posedness and nonlinear smoothing for dispersive equations through frequency--restricted estimates,} arXiv:2302.03575.

\bibitem{CS1} S. Correia, and J. D. Silva, {\it Nonlinear smoothing for dispersive PDE: A unified approach,} J. Differential Equations 269 (2020), no. 5, 4253--4285. 

\bibitem{cst} J. Colliander, G. Staffilani, and H. Takaoka, {\it Global wellposedness for KdV below $L^2$,} Math. Res. Lett. 6 (1999), no. 5--6, 755--778.

\bibitem{egt} M. B. Erdo\u gan,  T. B. G\"urel, and N.~Tzirakis, {\it The derivative nonlinear Schr\"odinger equation on the half line,} Annales de l'Institut Henri Poincar\'e (C) Analyse Non-Lin\'eaire. 35 (2018), 1947--1973. 

\bibitem{es} M. B. Erdo\u gan,  and G. Shakan,  {\it Fractal solutions of dispersive partial differential equations on the torus,} Selecta Mathematica 25 (1)  (2019), 1--26.
\bibitem{et1}   M.~B.~Erdo\u gan and N.~Tzirakis,  {\it Global smoothing for the periodic KdV evolution}, Int. Math. Res. Not.  (2013), no. 20, 4589--4614.

\bibitem{et3}   M.~B.~Erdo\u gan and N.~Tzirakis, {\it  Smoothing and global attractors for the Zakharov system on the torus}, Analysis \& PDE 6--3 (2013), 723--750.

\bibitem{et2} M.~B.~Erdo\u gan and N.~Tzirakis,  {\it Talbot effect for the cubic non-linear Schr\"odinger equation on the torus,} Mathematical Research Letters, (2013) 20 (6), 1081--1090.

\bibitem{et4} M.~B.~Erdo\u gan and N.~Tzirakis, {\it Long time dynamics for forced and weakly damped KdV on the torus,} Communications in Pure and Applied Analysis, v.12 (2013)  p.2669. 

\bibitem{etbook} M.~B.~Erdo\u gan and N.~Tzirakis, {\it Dispersive partial differential equations. Wellposedness and applications,} London Mathematical Society Student Texts, 86. Cambridge University Press, Cambridge, 2016. xvi+186 pp.

\bibitem{go} R. H. J Grimshaw, L. A. Ostrovsky, V. I. Shrira,  and Yu. A.  Stepanyants, {\it Long Nonlinear Surface and Internal Gravity Waves in a Rotating Ocean,}  Surveys in Geophysics, vol. 19 (4), (1998) 289--338.

\bibitem{isa} P. Isaza, and J. Mej\'{i}a, {\it Local well--posedness and quantitative ill--posedness for the Ostrovsky equation,} Nonlinear Analysis 70 (2009), 2306--2316.

\bibitem{isa1} P. Isaza, and J. Mej\'{i}a, {\it Global Cauchy problem for the Ostrovsky equation,} Nonlinear Analysis 67 (2007), no. 5, 1482--1503.

\bibitem{imt} P. Isaza,  J. Mej\'{i}a, and N. Tzvetkov, {\it A smoothing effect and polynomial growth of the Sobolev norms for the KP-II equation,} J. Differential Equations, 220 (2006), 1--17.

\bibitem{lhy} Y. S. Li, J. H. Huang, and W. Yan, {\it The Cauchy problem for the Ostrovsky equation with negative dispersion at the critical regularity,} J. Differ. Equ. 259 (2015), 1379--1408.

\bibitem{ostr} L. A. Ostrovsky, {\it Nonlinear internal waves in a rotating ocean,} Oceanology 18 (1978), 119--125.

\bibitem{step} Y. A. Stepanyants, {\it Nonlinear Waves in a Rotating Ocean (The Ostrovsky equation and its generalizations and applications),}Izvestiya, Atmospheric and Oceanic Physics 56, (2020), 16--32.


\bibitem{taobook} T. Tao, {\it Nonlinear dispersive equations. Local and global analysis, } CBMS Regional Conference Series in Mathematics 106,  AMS, Providence, RI, 2006. 

\bibitem{tsug} K. Tsugawa, {\it Well--posedness and weak rotation limit for the Ostrovsky equation,,} J. Differential Equations 247 (2009), 3163--3180.

\bibitem{yanyang} W. Yan, Y. Li, J. Huang, and J. Duan, {\it The Cauchy problem for the Ostrovsky with positive dispersion,} NoDEA Nonlinear Differential Equations Appl. 25 (2018), no. 3, Paper No. 22, 37 pp.

\end{thebibliography}
\end{document}